\documentclass[twoside, a4paper, 11pt, ]{amsart}

\usepackage[centertags]{amsmath}
\usepackage{amsfonts}
\usepackage{amssymb}
\usepackage{amsthm}
\usepackage{newlfont}
\usepackage{amscd}
\usepackage{amsmath,amscd}

\usepackage[arrow,matrix]{xy}
\usepackage{verbatim}
\usepackage{url}

\DeclareSymbolFont{rsfs}{U}{rsfs}{m}{n}
\DeclareSymbolFontAlphabet{\mathcal}{rsfs}

\DeclareFontEncoding{OT2}{}{} 
\DeclareTextFontCommand{\textcyr}{\fontencoding{OT2}\fontfamily{wncyr}\fontseries{m}\fontshape{n}\selectfont}
\newcommand{\Sha}{\textcyr{Sh}}
\newcommand{\Be}{\textcyr{B}}
\newcommand{\Ch}{\textcyr{Ch}}

\newcommand{\QQ}{\mathbb{Q}}
\newcommand{\ZZ}{\mathbb{Z}}
\newcommand{\be}{\begin{equation}}
\newcommand{\ee}{\end{equation}}

\newtheorem{thm}{Theorem}[section]

\newtheorem{lem}[thm]{Lemma}
\newtheorem{prop}[thm]{Proposition}

\theoremstyle{definition}
\newtheorem{defn}[thm]{Definition}
\newtheorem{subsec}[thm]{}
\theoremstyle{remark}
\newtheorem{rem}[thm]{Remark}

\DeclareMathOperator{\Ker}{ker}
\DeclareMathOperator{\Gal}{Gal}
\DeclareMathOperator{\Coker}{coker}

\DeclareMathOperator{\Br}{Br}
\DeclareMathOperator{\loc}{loc}

\def\Gm{{\mathbb{G}_m}}

\def\kb{{\bar k}}

\def\multmap{{\mu}}

\def\cV{{\cal V}}

\def\sV{{\cV}}

\def\textrm#1{\text{\rm #1}}

\def\Gal{\textrm{Gal}}

\def\loc{\textup{loc}}
  \def\Hom{\textrm{Hom}}
\def\im{\textrm{im\,}}      \def\int{\textrm{int}}
\def\inv{\textup{inv}}      
\def\ker{{\textup{ker}\,}}

   \def\tor{^{\textrm{tor}}}
\def\red{^{\textrm{red}}}     

\def\sss{^{\textrm{ss}}}     \def\uu{^{\textrm{u}}}
      
         \def\mult{^{\textrm{mult}}}

\def\Pic{\textup{ Pic\,}}
\def\Br{\textup{Br}\,}
\def\Brs{\textup{Br}}
\def\ev{\textup{ev}}
\def\Brn{\textup{ Br}_1}
\def\Brz{\textup{Br}_0}
\def\Bra{\textup{Br}_{\textup{a}} }

\def\uu{\textup{u}}
\def\red{\textup{red}}
\def\sss{\textup{ss}}

\def\tor{\textup{tor}}
\def\mult{\textup{mult}}

\font\cyr=wncyr10 scaled \magstep1
\def\Bcyr{\text{\cyr B}}

\def\mWS{\textup{ob}_S}

\newcommand{\labelto}[1]{\xrightarrow{\makebox[1.5em]{\scriptsize ${#1}$}}}

\def\kbar{{\bar{k}}}
\def\G{{\mathbb{G}}}

\def\Gbar{{\overline{G}}}
\def\Xbar{{\overline{X}}}

\def\sX{{\mathbb X}}
\def\XX{{\sX}}
\newcommand{\into}{\hookrightarrow}
\newcommand{\isoto}{\overset{\sim}{\to}}

\begin{document}

\title[Weak approximation]
{A cohomological obstruction to weak approximation for homogeneous spaces
}
\author{Mikhail Borovoi and Tomer M. Schlank}
\address{Borovoi: Raymond and Beverly Sackler School of Mathematical Sciences, Tel Aviv University, Tel Aviv 69978, Israel}
\thanks{This research was partially supported by the
Hermann Minkowski Center for Geometry
and by ISF grant 807/07}
\email{borovoi@post.tau.ac.il}
\address{Schlank: Institute of Mathematics, Hebrew University, Giv'at-Ram, Jerusalem 91904, Israel}
\email{tomer.schlank@gmail.com}

\maketitle

\begin{abstract}
Let $X$ be a homogeneous space, $X = G/H$, where $G$ is a connected linear algebraic group over a number field $k$,
and $H \subset G$ is a $k$-subgroup (not necessarily connected). Let $S$ be a finite set of places of $k$.
We compute a Brauer-Manin obstruction to weak approximation for $X$ in $S$ in terms of Galois cohomology.
\end{abstract}

\tableofcontents

\section{Introduction}\label{s:intro}

Let $X$ be an algebraic variety over a number field $k$,
and let $S$ be a finite set of places of $k$.
We say that $X$ has the \emph{weak approximation property in $S$}
if $X(k)$ is dense in $\prod_{v\in S} X(k_v)$ with respect to the
diagonal embedding.
 We say that $X$ has the \emph{weak approximation property} if
it has the weak approximation property in $S$ for any
finite set $S$ of places of $k$.

In 1970 Manin \cite{Manin} introduced a
general obstruction to the Hasse principle for a $k$-variety $X$,
using the Brauer group of $X$.
Using Manin's ideas, Colliot-Th\'{e}l\`{e}ne and Sansuc ~\cite{CTS} defined a Brauer-Manin obstruction
to weak approximation for $X$ in all $S$ simultaneously
(see also \cite{Sk}, \S5.2).
We consider a variation of this obstruction introduced in \cite{Bor96},
which is a Brauer-Manin obstruction to weak approximation
in a {specific} set of places $S$ for a $k$-variety $X$ having a $k$-point.

Write $k_S = \prod_{v \in S}k_v,$ then
$$X(k_S) = \prod \limits_{v \in S}X(k_v).$$
We assume that $X(k)\neq \emptyset$.
The Brauer-Manin obstruction of \cite{Bor96} is a map
$$
\mWS \colon X(k_S)\to \Be_{S,\emptyset}(X)^D,
$$
where $\Be_{S,\emptyset}(X):=\Be_S(X)/\Be(X)$ is a certain subquotient of the Brauer group $ \Br X$
(for details see \S\ref{sec:prelim-Brauer})
and ${}^D$ denotes the dual group, i.e. $\bullet^D = \Hom(\bullet,\QQ / \ZZ)$.
The map $\mWS$ is an obstruction in the following sense: if $x_S \in X(k_S)$ and $\mWS(x_S) \neq 0$, then $x_S$
is not contained in the closure of $X(k)$ in $X(k_S)$.
In particular, if the map $\mWS$ is not identically 0, then  $X$ does not have weak approximation in $S$.

Now let $X$ be a homogeneous space of a connected linear $k$-group $G$.
It is convenient to use the notion of a quasi-trivial group, introduced by Colliot-Thelene,
see   Definition 2.1 in \cite{CT06}  or Definition ~\ref{def:QT_grp} below.
By  Lemma ~\ref{l:QT_not_resrict} below we may assume that $X$ is a homogeneous space of a \emph{quasi-trivial} $k$-group $G$.
In ~\cite{Bor:cohomological-obstruction} we computed the Brauer-Manin obstruction  of \cite{Bor96}
to the Hasse principle for $X$ in terms of Galois cohomology.
Here we do a similar computation for the Brauer-Manin obstruction of \cite{Bor96}
to weak approximation for $X$.
We assume that $X$ has a $k$-point $x^0$.
Let $H$ denote the stabilizer of $x^0$ in $G$, then $X=G/H$.

From now on we  assume that $X=G/H$, where $G$ is quasi-trivial (and $H$ is not necessarily connected).
We  describe the group $\Be_{S,\emptyset}(X)^D$ and the map $\mWS$ in terms of Galois cohomology.
Let $H^{\mult}$ denote the greatest quotient of $H$ that is a group of multiplicative type.
Write
$$H^1(k_S,H^{\mult}) = \prod_{v \in S}H^1(k_v,H^{\mult})$$
and set
$$\Ch_S^1(k,H^{\mult}) = \Coker[H^1(k,H^{\mult}) \xrightarrow{\loc_S}H^1(k_S,H^{\mult}) ],$$
where $\loc_S$ is the localization map.
The following theorem describes the group  $\Be_{S,\emptyset}(X)^D$ in terms of Galois cohomology.

\begin{thm} [Theorem ~\ref{t:main_gen}]\label{thm:Be_intro}
Let $X=G/H$, where $G$ is a quasi-trivial $k$-group over a number field $k$,
and $H$ is a $k$-subgroup of $G$.
Then there is a canonical isomorphism:
$$ \phi \colon \Be_{S,\emptyset}(X)^D \isoto \Ch^1_S(k,H^{\mult}).$$
\end{thm}

We  wish to describe not only the group $\Be_{S,\emptyset}(X)^D$, but also
the map $\mWS$ in terms of Galois cohomology.
Let $x_S = (x_v)_{v\in S} \in X(k_S)$. For $v \in S$ consider the $G(k_v)$-orbit $G(k_v). x_v$ of $x_v$ in $X(k_v)$.
This orbit defines a cohomology class
$$\xi_v(x_v) \in \Ker [H^1(k_v,H) \to H^1(k_v,G)],$$
cf. \cite[\S I.5.4]{Serre:CG}.
Consider the canonical epimorphism $\mu\colon H \to H^{\mult},$ and set
\begin{align*}
\xi^{\mult}_v(x_v) &= \mu_*(\xi_v(x_v)) \in H^1(k_v,H^{\mult}),\\
\xi^{\mult}_S(x_S) &= \left(\xi^{\mult}_v(x_v)\right)_{v \in S} \in H^1(k_S,H^{\mult}).
\end{align*}
Let $c_S(x_S)$ denote the image of $\xi^{\mult}_S(x_S)$ under the canonical map
$$ H^1(k_S,H^{\mult}) \to \Ch^1_S(k,H^{\mult}).$$
The following theorem describes the map  $\mWS$ in terms of Galois cohomology.

\begin{thm} [Theorem ~\ref{t:main_gen}] \label{t:main_intro}
Let $k$, $G$, $H$, and $X$ be as in Theorem \ref{thm:Be_intro}.
The following diagram commutes:
$$
\xymatrix{
X(k_S)\ar@{=}[d] \ar[r]^-{\mWS} & \Be_{S,\emptyset}(X)^D \ar[d]^{\phi}_{\cong}\\
X(k_S)\ar[r]^-{-c_S} & \Ch_S^1(k,H^{\mult}).
}
$$
\end{thm}

Note that results similar to our Theorem \ref{t:main_intro}
in the special case of a {\em finite} group $H$
were obtained earlier by D.~Harari \cite{Harari}, \S 5, Theorem 4,
and C.~Demarche \cite{Demarche}, \S 7, Corollary 4.

\begin{rem}
In this paper we consider the Brauer-Manin obstruction $\mWS$ of \cite{Bor96},
which is a variation of the Brauer-Manin obstruction associated with the group $\Be_\omega(X)$,
cf. \cite[(6.2.3)]{Sansuc}.
By \cite[Thm. 2.4]{Bor96}, if $H$ is connected or abelian, then  $\mWS$
is the only obstruction to weak approximation in $S$ for $X$.
However, it may be not the only obstruction if $H$ is non-connected and non-abelian,
see Demarche \cite[\S 6, Prop.~2]{Demarche}.
The failure of weak approximation in the counter-example of Demarche
can be explained by the  Brauer-Manin obstruction associated with
the larger group $\textup{Br}_{\textup{nr},1}(X)$, {\em loc. cit.}
\end{rem}

Now let $H$ be an arbitrary $k$-group of multiplicative type. Let
$\widehat{H}:=\Hom_\kbar(H_\kbar,\G_{m,\kbar})$
be the (geometric) character group of $H$.
Let $\Sha_S^1(k,\widehat{H})$ denote the kernel of the localization map
$$\loc^1_{S^\complement} = \loc^1_{\cV(k) \smallsetminus S}\colon H^1(k,\widehat{H}) \to \prod \limits_{v \in \cV(k)\smallsetminus S}H^1(k_v,\widehat{H})$$
where $\cV(k)$ is the set of all places of $k$.
As a byproduct of our study of the group $\Ch^1(k, H^{\mult})$
we  obtain a duality theorem:

\begin{thm}[Theorem ~\ref{t:dual_t}]
\label{t:dual_intro}
Let $H$ be a $k$-group of multiplicative type over a number field $k$.
There is a canonical non-degenerate pairing
$$
\Sha_S^1(k,\widehat{H})/\Sha_{\emptyset}^1(k,\widehat{H}) \times
\Ch^1_S(k,H) \xrightarrow{\cup_S} \QQ / \ZZ.
$$
\end{thm}
Theorem ~\ref{t:dual_intro} generalizes a result of Sansuc \cite{Sansuc}, Lemma 1.4, who considered the case of finite $H$.
This theorem  can be also deduced from the Poitou-Tate exact sequence for groups of multiplicative type,
see \cite[Thm.~6.3]{Demarche-Suites}.

\bigskip

\noindent\emph{Acknowledgements.}
The authors are grateful to David Harari for useful discussions and to the referee for helpful comments.
The first-named author worked on the paper while visiting the Max-Planck-Institut f\"ur Mathematik, Bonn;
he thanks the Institute for hospitality, support, and excellent working conditions.
\bigskip

\noindent\emph{Notation.}

By $k$ we denote a field of characteristic 0, and by $\kbar$ a fixed algebraic closure of $k$.
By a $k$-variety we mean a separated geometrically integral scheme of finite type over $k$.
For a $k$-variety $X$
we set $\Xbar=X\times_k \kbar$ and
$$
U(\Xbar)=\kbar[X]^*/\kbar^*,
$$
where $\kbar[X]$ is the ring of regular functions on $\Xbar$ and $\kbar[X]^*$
is the group of invertible regular functions on $\Xbar$.
We denote by $\Pic \Xbar$ the Picard group of $\Xbar$.
Note that both $U(\Xbar)$ and $\Pic \Xbar$ are Galois modules,
i.e. the Galois group $\Gal(\kbar/k)$ acts on them.

By an algebraic $k$-group $H$ we mean a \emph{linear} algebraic group (not necessarily connected).
We write $\widehat{H}$ or $\XX(H)$ for the (geometric) character group of $H$, i.e.
$$
\XX(H)=\widehat{H}:=\Hom_\kbar(\overline{H},\G_{m,\kbar}).
$$

When $k$ is assumed to be a number field, we write $\sV(k)$ for the set of places of $k$.
If $v\in\sV(k)$, we write $k_v$ for the completion of $k$ at $v$.
Let $S\subset \sV(k)$ be a finite set of places of $k$.
We set $k_S=\prod_{v\in S}k_v$, then for a $k$-variety $X$ we have
$$
X(k_S)=\prod_{v\in S} X(k_v).
$$
The set of $k$-points $X(k)$ embeds diagonally into $X(k_S)$, and
we denote by $\overline{X(k)}$ the closure of $X(k)$ in $X(k_S)$.
If $G$ is a  linear algebraic group over $k$, we set
$$
H^1(k_S, G):=\prod_{v\in S} H^1(k_v,G).
$$

\section{Preliminaries on the Brauer group and  the Brauer-Manin obstruction}\label{sec:prelim-Brauer}
Let $k$ be a field of characteristic 0.
Let $X$ be  a smooth geometrically integral $k$-variety with a marked $k$-point $x^0$.
Then $\Br X:= H^2_{\text{\'et}}(X, \Gm)$ is the cohomological Brauer group of $X$.
We  use the following notation:
\begin{gather*}
\Brz X =\im[\Br k\to \Br X];\quad
\Brn X=\Ker[\Br X\to \Br X_\kb];\\
\Brs_{x^0}X  = \Ker [(x^0)^*\colon \Brn X\to\Br k];\quad \Bra X=\Brn X/\Brz X.
\end{gather*}
We have
$$\Brn X = \Brz X \oplus \Brs_{x^0}X,$$
and therefore we have a canonical isomorphism $\Brs_{x^0} X\isoto\Bra X$.

If $S \subset \cV(k)$ is a finite subset, let  $\Bcyr_S(X)$ be the subgroup of $\Brs_{x^0} X$ consisting of
elements $b$ whose localizations $\loc_v b$
in $\Br X_{k_v}$ are trivial for all
places $v$ of $k$ outside $S$.
We set $\Be(X)=\Be_\emptyset(X)$.

Note that  for any $S$ we have $\Be_{\emptyset}(X) \subset \Be_{S}(X)$.
We denote
$$\Be_{S,\emptyset}(X): = \Be_{S}(X) / \Be_{\emptyset}(X)=\Be_{S}(X) / \Be(X).$$

Now we describe the Brauer-Manin obstruction of \cite{Bor96} to weak approximation in $S$ for $X$.

For our purposes the Brauer-Manin obstruction coming from the subgroup $\Brn X$
of $\Br X$ will suffice.
Following \cite{Sansuc} and \cite{Bor96}, we define the Brauer--Manin obstruction in terms
of the group  $\Be_{S,\emptyset}(X)$.

Let $X$ be a smooth geometrically integral variety over a field $k$ of characteristic 0.
Consider the pairing
$$ X(k)\times \Brn X\xrightarrow{\ev} \Br k $$
where $\ev$ is the evaluation map $\ev\colon (x,b)\mapsto b(x)$.
This pairing is additive in $b$:
$$ \ev(x,b+b')=\ev(x,b)+\ev(x,b').$$
If $k$ is a local field,
then the above pairing
gives us a pairing
$$ X(k)\times \Brn X \xrightarrow{\inv\, \circ\, \ev} \QQ/\ZZ,$$
$$ (x,b)\mapsto \inv (b(x)) $$
where $\inv\colon \Br k \to \QQ/\ZZ$ is the homomorphism of local class field theory.
This pairing is continuous in $x$ (see \cite{Sansuc}, Lemma 6.2, or \cite{BD}, Lemma 6.2)
and is additive in $b$.

Now let $k$ be a number field.
Let $S\subset\cV(k)$ be a finite subset.
Consider the pairing
\begin{gather}\label{e:def_br_obs}
\langle\,,\,\rangle_S\colon X(k_S)\times \Bcyr_S(X)\to \QQ / \ZZ,\\
\langle(x_v)_{v\in S}, b\rangle_S = \sum_{v\in S} \big(\inv_v (b(x_v))\big).\notag
\end{gather}
If $b\in \Be_\emptyset(X)$, then $\langle x_S,b \rangle_S=0$ for any $x_S\in X(k_S)$.
Thus the pairing $\langle\,,\,\rangle_S$ induces a pairing
\be\label{e:manin2}
X(k_S)\times \Be_S(X)/\Be_\emptyset(X)\to \QQ / \ZZ\;,
\ee
which is additive in the second argument.
We call the pairings \eqref{e:def_br_obs} and \eqref{e:manin2} the {\em Manin pairings}.
Thus we obtain a map
$$\mWS\colon X(k_S)\to (\Be_S(X)/\Be_\emptyset(X))^D= \Be_{S,\emptyset}(X)^D\;.$$
The map $\mWS$ is continuous because the pairing $\langle\,,\,\rangle_S$ is continuous
in $x_S$.
Further, by the Hasse-Brauer-Noether theorem if $x\in X(k)\subset X(k_S)$, then $\mWS(x)=0$.
It follows that if $x_S$ is contained in the closure $\overline{X(k)}$
of $X(k)$ in $X(k_S)$, then $\mWS(x_S)=0$.

Recall  that $X$ has the weak approximation property in $S$,
if $X(k)$ is dense in $X(k_S)$.
If $X$ has the weak approximation property in $S$, then $\overline{X(k)}=X(k_S)$
and  $\mWS$ is identically 0.
We see therefore that $\mWS$ is an obstruction to weak approximation in $S$.
We call $\mWS$
{\em the Brauer--Manin obstruction to weak approximation in $S$,
associated with $\Be_{S,\emptyset}$.}

The obstruction $\mWS$ is functorial. Namely,
let $\pi\colon (X, x^0)\to (Y,y^0)$ be a morphism of $k$-varieties with marked $k$-points.
Then the following diagram is commutative:
$$
\xymatrix{
X(k_S) \ar[d]^{\pi} \ar[r]^-{\mWS}  &\Be_{S,\emptyset}(X)^D \ar[d]^{\pi_*} \\
Y(k_S) \ar[r]^-{\mWS}  &\Be_{S,\emptyset}(Y)^D }
$$
where $\pi_*$ is the homomorphism induced by $\pi$.

\begin{rem}
In \cite{Bor96} the group $\Be_S(X)$ and the Brauer-Manin obstruction were defined in terms of $\Bra X=\Brn X/\Brz X$
rather than $\Brs_{x^0} X$.
However, these two groups are canonically  isomorphic,
and one can check that
our definitions here are essentially equivalent to those of \cite{Bor96}.
Note that
though we defined the obstruction $\mWS$ using $x^0$,  this obstruction essentially
does not depend on $x^0$, cf. \cite{Bor96}, \S1.
\end{rem}

\section{Preliminaries on quasi-trivial groups}\label{s:category}

\def\pn{\par\noindent}

All the lemmas in this section are well known.
For the reader's convenience we provide short proofs and/or references.

\begin{subsec}\label{subsec:groups}
{\em Notation concerning  linear algebraic groups.}
Let $G$ be a connected linear algebraic group
over a field $k$ of characteristic 0.
We use the following notation:
\pn $G^\uu$ is the unipotent radical of $G$;
\pn $G^\red=G/G^\uu$, it is a reductive $k$-group;
\pn $G^\sss$ is the commutator subgroup of $G^\red$,
it is a semisimple $k$-group;
\pn $G^\tor=G^\red/G^\sss$, it is a $k$-torus.

Let  $H$ be a linear $k$-group, not necessarily connected.
We denote by $H^\mult$ the biggest quotient group of $H$ that is a group of multiplicative type.
Let $\widehat{H}$ denote the character group of $H$. We have $\widehat{H}=\widehat{H^\mult}$.
If $H$ is connected, then $H^\mult=H^\tor$.
\end{subsec}

Recall that a torus $T$ is called quasi-trivial if its character group $\widehat{T}$ is a permutation Galois module.

\begin{defn}[\cite{CT06}, Proposition 2.2]\label{def:QT_grp}
Let $k$ be a field and $G$ a connected linear $k$-group. We say that $G$ is \emph{quasi-trivial}
if $G^{\tor}$ is quasi-trivial and $G^{\sss}$ is simply connected.
\end{defn}

Let $X$ be a homogeneous space of a connected linear $k$-group $G$.
By virtue of the following lemma one can always take $G$ to be quasi-trivial without changing $X$.

\begin{lem}\label{l:QT_not_resrict}
Let $G$ be a connected linear algebraic group over a field $k$  of characteristic 0, then there exists
a surjective homomorphism $G'\to G$ such that $G'$ is a quasi-trivial $k$-group.
\end{lem}

\begin{proof}
See \cite{CT06}, Proposition-Definition 3.1.
\end{proof}

\begin{subsec} \label{subsec:pairs-of-groups}
 By a pair of $k$-groups we mean  a pair $(G,H)$, where $G$ is a \emph{quasi-trivial} $k$-group
 and $H$ is a $k$-subgroup of $G$ (not necessarily connected).
 A pair $(G,H)$ defines a homogeneous space $X:=G/H$ together with a marked point $x^0=eH\in X(k)$,
 where $e\in G(k)$ is the identity element of $G$.

 By a morphism of pairs $\phi\colon (G_1,H_1)\to (G_2,H_2)$ we mean a {\em surjective} homomorphism $\phi\colon G_1\to G_2$
 such that $\phi(H_1)=H_2$.
 If we set $X_1=G_1/H_1$ and $X_2=G_2/H_2$, then we have an induced   morphism $\phi_*\colon (X_1, x^0_1)\to (X_2,x^0_2)$,
 where $x^0_1$ and $x^0_2$ are the corresponding marked points.
\end{subsec}

Let $G$ be a $k$-group. In the next lemma we  use the notation $\sX(G) = \widehat{G}$
to denote the Galois module of (geometric) characters of $G$.

\begin{prop}\label{l:H_k_pic_u}
Let $(G,H)$ be a pair of $k$-groups with quasi-trivial $G$ as in \ref{subsec:pairs-of-groups}, and let $X=G/H$.
Let $\overline{X} = X\times_{k}\bar{k}$. Consider the natural morphism of Galois modules
$$ \sX(G) \to \sX(H) $$
 and the dual morphism of $k$-groups of multiplicative type
$$ H^{\mult} \to G^{\mult} = G^{\tor}$$
Then there are natural isomorphisms of Galois modules
\begin{enumerate}
\item\label{l:H_k_pic_u:it:U} $U(\overline{X}) \cong \Ker[\sX(G) \to \sX(H)] \cong \sX(\Coker[H^{\mult} \to G^{\tor}])$
\item\label{l:H_k_pic_u:it:Pic} $\Pic \overline{X} \cong \Coker[\sX(G) \to \sX(H)] \cong \sX(\Ker[H^{\mult} \to G^{\tor}])$
\end{enumerate}
\end{prop}

\begin{proof}
We have canonical isomorphisms
 $$\Ker[\sX(G) \to \sX(H)] \cong \sX(\Coker[H^{\mult} \to G^{\tor}])$$
and
 $$\Coker[\sX(G) \to \sX(H)] \cong \sX(\Ker[H^{\mult} \to G^{\tor}]).$$

Proof of \eqref{l:H_k_pic_u:it:U}.
The natural map $G\to G/H=X$ induces an embedding $U(\Xbar)\into U(\Gbar)$.
By Rosenlicht's theorem (\cite{Rosenlicht}, Theorem 3) the injection $\XX(G)\into \kbar[\Gbar]^*$
induces an isomorphism $\XX(G)\cong U(\Gbar)$.
It is easy to see that a character $\chi\in\XX(G)$ corresponds to an element of $U(\Gbar)$ coming from $U(\Xbar)$
if and only if $\chi\in \ker[\XX(G)\to\XX(H)]$.

Proof of \eqref{l:H_k_pic_u:it:Pic}. See Popov ~\cite{Popov}, Theorem 4
(we use that $\Pic \overline{G} = 0$, cf. \cite{CT06}, Definition 2.1).
\end{proof}

\section{The cohomological obstruction $c_S$}\label{s:c_S}
In this section,
using nonabelian Galois cohomology we  define a cohomological obstruction to weak approximation in $X$
at a finite set of places $S \subset \cV(k)$.
This obstruction  takes values in the group $\Ch^1_S(k,H^{\mult})$, and we  denote it by $c_S$.

Let $k$ be a number field.
Let $\bar{k}$ be a fixed algebraic closure of $k$.
Let $(G,H)$ be a pair of $k$-groups as in \ref{subsec:pairs-of-groups} over a number field $k$.

\begin{lem}\label{l:orbits}
Let S be a finite subset of $\cV(k)$,
$(G,H)$ be a pair of $k$-groups as in \ref{subsec:pairs-of-groups},
and $x_S \in X(k_S)$ be a $k_S$-point.
Then $x_S \in \overline{X(k)}$ if and only if the orbit $G(k_S). x_S$ of $x_S$ contains a $k$-point of $X$.
\end{lem}

\begin{proof}
This is well known, see e.g. \cite{Bor:Toulouse}, \S 2.1.
\end{proof}

We see from Lemma ~\ref{l:orbits} that
a point $x_S \in X(k_S)$ lies in the closure $\overline{X(k)}$ of $X(k)$ if and only if
its $G(k_S)$-orbit $G(k_S). x_S$ lies in in the image of the localization map
\be
\label{e:orbits} \loc_S\colon G(k)\backslash X(k)  \to  G(k_S)\backslash X(k_S).
\ee

\begin{subsec}\label{ss:cohomology} {\em Cohomological formulation.}
By \cite{Serre:CG}, I.5.5, Corollary 1 of Proposition 36, we have a canonical  bijection
\be\label{e:orb_coho_2} \tau_k \colon G(k)\backslash X(k) \labelto{\sim} \Ker \left[H^1(k,H) \to H^1(k,G)\right].\ee
For a finite set $S$ of places of $k$ we obtain a bijection
\be
\label{e:orb_coho_10} \tau_S=\prod_{v\in S}\tau_{k_v} \colon G(k_S)\backslash X(k_S) \to \Ker \left[H^1(k_S,H) \to H^1(k_S,G)\right]
\ee
We have a commutative diagram with bijective vertical arrows:
$$
\xymatrix{
G(k) \backslash X(k) \ar[d]^-{\tau_k}   \ar[r]^{\loc_S}
& G(k_S) \backslash X(k_S) \ar[d]^-{\tau_S}\\
\Ker[H^1(k,H)\to H^1(k,G)] \ar[r]^-{\loc_S} & \Ker[H^1(k_S,H) \to H^1(k_S,G)].
}
$$
An element $x_S \in X(k_S)$ is contained in $\overline{X(k)}$ if and only if $\tau_S(G(k_S).x_S)$ lies in the image of the map
$$
\Ker[H^1(k,H)\to H^1(k,G)] \xrightarrow{\loc_S}  \Ker[H^1(k_S,H) \to H^1(k_S,G)].
$$
\end{subsec}

\begin{subsec}\label{ss:c_S} {\em The definition of $c_S$.}
Consider the following commutative diagram:
\be\label{e:big_coho}
\xymatrix{
X(k) \ar[d]^\nu \ar@{^{(}->}[r] & X(k_S) \ar[d]^{\nu_S}\\
G(k) \backslash X(k) \ar[d]^-{\tau_k}   \ar[r]^{\loc_S}
& G(k_S) \backslash X(k_S) \ar[d]^-{\tau_{S}}\\
\Ker[H^1(k,H)\to H^1(k,G)] \ar[d]\ar[r]^-{\loc_S} &\Ker[H^1(k_S,H) \to H^1(k_S,G)]\ar[d]\\
H^1(k,H) \ar[d]^{\multmap}\ar[r]^-{\loc_S} &H^1(k_S,H)\ar[d]^{\multmap_S}\\
H^1(k,H^{\mult}) \ar[r]^-{\loc_S} &H^1(k_S,H^{\mult}).
}
\ee
By composing the arrows in the right column of diagram ~\eqref{e:big_coho} we obtain a
map
$$\tilde{c}_S \colon  X(k_S) \to H^1(k_S,H^{\mult}).$$
By composing $\tilde{c}_S$ with the canonical map
$$H^1(k_S,H^{\mult})\to\Ch^1_S(H^{\mult}),$$
we obtain a map
$$c_S \colon X(k_S) \xrightarrow{\tilde{c}_S} H^1(k_S,H^{\mult}) \to \Ch_S^1(H^{\mult}).$$
We  prove that the map $c_S$ is indeed an obstruction to weak approximation.
\end{subsec}

\begin{lem}\label{l:c_is_obs}
Let $c_S$ be the map defined above and let $x_S \in \overline{X(k)} \subset X(k_S)$.
Then $c_S(x_S) = 0$.
\end{lem}

\begin{proof}
First note that the map $$c_S\colon X(k_S) \to \Ch_S^1(H^{\mult})$$ is constant on $G(k_S)$-orbits.
Since the $G(k_S)$-orbits are open, the map $c_S$ is continuous, and therefore we may assume that
$x_S \in X(k)$. Now the assertion follows from the
commutativity of diagram ~\eqref{e:big_coho}.
\end{proof}

Note that every step in the definition of $c_S$ is functorial, and
therefore $c_S$ is functorial as well. Namely:
\begin{lem}\label{l:c_is_func}
Let $\phi\colon (G_1, H_1)\to (G_2, H_2)$ be a morphism of pairs as in \ref{subsec:pairs-of-groups}
 over a number field $k$ (with quasi-trivial groups $G_1$ and $G_2$).
Set $X_1=G_1/H_1$ and $X_2=G_2/H_2$,
and let $\phi_X\colon X_1\to X_2$ be the induced map.
Since
$\phi(H_1) = H_2$, $\phi$ induces a map $\phi^{\multmap}\colon
H^{\mult}_1 \to H^{\mult}_2$, and the following diagram commutes:
$$
\xymatrix{
X_1(k_S) \ar[d]^{c_S} \ar[r]^-{\phi_X}  &X_2(k_S)\ar[d]^{c_S}\\
\Ch_S^1(H_1^{\mult}) \ar[r]^-{\phi^{\multmap}_*} &\Ch_S^1(H_2^{\mult}) .
}
$$
\qed
\end{lem}

\section{Proofs: the case of a torus}\label{s:torus}
In this section we prove our results in the special case when
our pair $(G,H)$ over $k$ is such that $G$ is
a quasi-trivial torus.
Note that in this case $H \subset G$ is a $k$-group of multiplicative type and
 $X = G/H$ has a canonical structure  of a $k$-torus.
 We write $T=G/H$ and denote by $e$ the neutral element of $T$.
 Then $e\in T(k)$ is our  marked point $x^0\in X(k)$.

Our main result in this special case is:
\begin{thm}\label{t:main_torus}
Let $(G,H)$ be a pair of $k$-groups over a number field $k$  such that $G$ is a quasi-trivial $k$-torus.
Set $T=G/H$.
Then $H = H^{\mult}$ is a $k$-group of multiplicative type, $T$ is a $k$-torus,
$T(k_S)$ and $\overline{T(k)} \subset T(k_S)$ are groups, and

(i) There are canonical isomorphisms
\begin{align*}
\beta_S\colon &\Sha^1_S(\widehat{H}) \isoto \Be_S(T),\\
\beta_{S,\emptyset} \colon &\Sha^1_{S,\emptyset}(\widehat{H}):=\Sha^1_S(\widehat{H}) /\Sha^1_\emptyset(\widehat{H})
\isoto \Be_{S,\emptyset}(T).
\end{align*}

(ii) We have a canonical commutative diagram
$$\label{e:tor_main}
\xymatrix{
T(k_S) / \overline{T(k)} \ar[d]^-{-c_S}_{\cong} \ar@{}[r]|{\times} & \Be_{S,\emptyset}(T) \ar[r]^-{\langle\,,\,\rangle_S} &\QQ / \ZZ \ar@{=}[d] \\
 \Ch^1_S(k,H) \ar@{}[r]|{\times} & \Sha^1_{S,\emptyset}(\widehat{H}) \ar[u]^{\beta_{S,\emptyset}}_{\cong} \ar[r]^-{\cup_S} &  \QQ / \ZZ,
}
$$
in which by abuse of notation we write
 $c_S\colon T(k_S) / \overline{T(k)}\to \Ch^1_S(k,H)$
for the map induced by the map $c_S$ defined in \S\ref{s:c_S},
and by abuse of notation again  we write $\langle\,,\,\rangle_S\colon T(k_S) / \overline{T(k)}\times\Be_{S,\emptyset}(T)\to\QQ/\ZZ$
for the pairing induced by the Manin pairing.
Moreover, in this diagram:
\begin{enumerate}
\item Both vertical arrows are isomorphisms of abelian groups.
\item Both pairings  are perfect pairings of finite abelian groups.
\end{enumerate}
\end{thm}

When proving Theorem ~\ref{t:main_torus} we rely  on Sansuc ~\cite{Sansuc}, \S8.
After we establish our result we shall be able to prove the following duality theorem:

\begin{thm}\label{t:dual_t}
Let H be a group of multiplicative type over a number field $k$.
Then there is a non-degenerate pairing of finite abelian groups
$$\Sha_{S,\emptyset}^1(k,\widehat{H}) \times
\Ch^1_S(k,H) \xrightarrow{\cup_S} \QQ / \ZZ ,$$
where
$$ \cup_S = \sum\limits_{v \in S}\inv_v\circ \cup_{k_v}$$
and $\cup_{k_v}$ is the cup-product pairing:
$$ \cup_{k_v} \colon H^1({k_v},\widehat{H}) \times H^1({k_v},H) \xrightarrow{\cup} H^2({k_v},\mathbb{G}_m)=\Br k. $$
\end{thm}

Our first step in proving Theorem ~\ref{t:main_torus} is to describe the Brauer group of $T$ in terms of
the second Galois cohomology of the Galois module $\widehat{T}$, and to describe the Manin pairing in terms of the cup product.

\begin{lem}\label{l:h2_br}
Let $T$ be a torus defined over a field $\mathbb{F}$ of characteristic 0.
Then there is a canonical commutative diagram
\be\label{e:h2_br}\xymatrix{
T(\mathbb{F}) \ar@{=}[d] \ar@{}[r]|-{\times}  &\Brs_e T \ar[r]^-{\ev} & \Br \mathbb{F} \ar@{=}[d] \\
T(\mathbb{F})\ar@{}[r]|-{\times}  &H^2(\mathbb{F},\widehat{T}) \ar[u]^{\theta_{\mathbb{F}}}_{\cong} \ar[r]^-{\cup} & \Br \mathbb{F}
}\ee
where $\ev$ is the evaluation map $(t,b) \mapsto b(t)$, and $\theta_{\mathbb{F}}$
is the canonical isomorphism of abelian groups from \cite{Sansuc}, Lemma 6.9(ii).
Both pairings in this diagram are additive in both arguments.
\end{lem}

Note that the additivity of the pairing $\ev$ in the first argument means that
$$
\ev(tt',b)=\ev(t,b)+\ev(t',b) \quad \text{for } t,t'\in T(\mathbb{F}), b\in \Brs_e T.
$$

\begin{proof}
This is the upper half of the commutative diagram (8.11.2) in the proof of ~\cite{Sansuc}, Lemma 8.11.
Note that $\theta_{\mathbb{F}}$ is the map
$\theta_{\mathbb{F}}\colon H^2(\mathbb{F},\widehat{T}) = H^2(\mathbb{F},U(\overline{T})) \to \Bra T = \Brs_{e}T$
that appears in the long exact sequence in ~\cite{Sansuc}, Lemma 6.3(ii).
Since the bottom pairing is additive in both arguments, so is the top one.
\end{proof}

\begin{subsec}\label{subsec:br_fin}
Now let $k$ be a number field and  $S \subset \cV(k)$ be a finite set of places.
Then diagram \eqref{e:h2_br} above can be used in order to compute the Manin pairing.

First we note that for every $v \in S\subset \cV(k)$ there is a canonical inclusion $\inv_v \colon \Br k_v \hookrightarrow \QQ / \ZZ $.
Thus for every $v \in S$  we obtain a commutative diagram
\be\label{e:h2_br_loc}\xymatrix{
T(k_v) \ar@{=}[d] \ar@{}[r]|-{\times}  &\Brs_e T_{k_v} \ar[r]^-{\inv_v\, \circ\,  \ev_v} & \QQ/\ZZ \ar@{=}[d] \\
T(k_v) \ar@{}[r]|-{\times}  &H^2(k_v,\widehat{T}) \ar[u]^{\theta_v}_{\cong} \ar[r]^-{\inv_v\, \circ\, \cup_v} & \QQ/\ZZ ,\\
}\ee
where $\theta_v := \theta_{k_v}$.

Now we sum several copies of diagram ~\eqref{e:h2_br_loc} (one for each $v \in S$) and obtain a commutative diagram
 \be\label{e:h2_br_loc_multy}\xymatrix{
T(k_S) \ar@{=}[dd] \ar@{}[r]|-{\times}  &\prod_{v \in S} \Brs_e T_{k_v} \ar[rr]^-{\sum \inv_v\circ \ev_v}& & \QQ/\ZZ \ar@{=}[dd] \\
& & & \\
T(k_S) \ar@{}[r]|-{\times}  &H^2(k_S,\widehat{T}) \ar[uu]^{\sum \limits_{v \in S} \theta_v}_{\cong} \ar[rr]^-{\cup_S}& & \QQ/\ZZ .
}\ee
Since the isomorphism $\theta_{\mathbb{F}}$ is functorial in $\mathbb{F}$, it induces natural isomorphisms
\begin{align}\label{e:iso_sha2_be}
\theta_S \colon \Sha^2_S(\widehat{T}) \isoto &\Be_S(T),\\
\theta_{S,\emptyset} \colon \Sha^2_{S,\emptyset}(\widehat{T}) \isoto &\Be_{S,\emptyset}(T).
\end{align}

Using the natural homomorphisms
$$\loc_S\colon \Sha^2_S(\widehat{T}) \to H^2(k_S,\widehat{T}),\quad \Be_S(T) \to \prod \limits_{v \in S} \Brs_e T_{k_v},$$
we obtain from diagram \eqref{e:h2_br_loc_multy} a commutative diagram
$$
\xymatrix{
T(k_S) \ar@{=}[d] \ar@{}[r]|-{\times}  &\Be_S(T) \ar[r]^-{\ev_S} & \QQ/\ZZ \ar@{=}[d] \\
T(k_S) \ar@{}[r]|-{\times}  &\Sha_S^2(\widehat{T}) \ar[u]^{\theta_S}_{\cong} \ar[r]^-{\cup_S} & \QQ/ \ZZ \\
}
$$
where $\theta_S$ is  isomorphism \eqref{e:iso_sha2_be} and
$$
\ev_S(t_S,b)= \sum_{v\in S} \inv_v(b(t_v))\text{ for }b\in\Be_S(T) \subset \Br_e T \subset \Brn T \text{ and } t_S\in T(k_S).
$$

But $\ev_S$ is exactly the Manin pairing ~\eqref{e:def_br_obs} (denoted by $\langle\,,\,\rangle_S$ there).
Thus we obtain the following commutative diagram, containing the map induced by the Manin pairing as the top pairing:
\be
\label{e:h2_br_fin}\xymatrix{
T(k_S)/\overline{T(k)} \ar@{=}[d] \ar@{}[r]|-{\times}  &\Be_{S,\emptyset}(T) \ar[r]^-{\langle\, ,\, \rangle_S} & \QQ/\ZZ \ar@{=}[d] \\
T(k_S)/\overline{T(k)} \ar@{}[r]|-{\times}  &\Sha_{S,\emptyset}^2(\widehat{T}) \ar[u]^{\theta_S}_{\cong} \ar[r]^-{\cup_S} & \QQ/ \ZZ \\
}
\ee
Note that we may write $T(k_S)/\overline{T(k)}$ instead of $T(k_S)$, because we know that the Manin pairing
vanishes on $\overline{T(k)}$ and is additive in the first argument.
\end{subsec}

The next step in our proof will be using connecting maps in order to reduce the second Galois cohomology that appears
in diagram ~\eqref{e:h2_br_fin} to the first Galois cohomology.

\begin{lem}\label{l:cup_lem}
Let $1 \to H \to G \to T \to 1$ be a short exact sequence of
groups of multiplicative type over a field $\mathbb{F}$ of characteristic 0,
and let $0 \to \widehat{T} \to \widehat{G} \to \widehat{H} \to 0$ be the dual exact sequence. Then
the following diagram anti-commutes:
\be\label{e:h1_h2}\xymatrix{
H^0(\mathbb{F},T) \ar[d]^{\partial} \ar@{}[r]|-{\times}  &H^2(\mathbb{F},\widehat{T}) \ar[r]^-{\cup} & \Br(\mathbb{F}) \ar@{=}[d]\\
H^1(\mathbb{F},H) \ar@{}[r]|-{\times}  &H^1(\mathbb{F},\widehat{H})\ar[u]^{\partial} \ar[r]^-{\cup} & \Br(\mathbb{F}) \\
}\ee
\end{lem}

\begin{proof}
The proof is similar to that of  ~\cite{Sansuc}, Lemma 8.11.
Let $t\in T(\mathbb{F})$ and $f \in H^1(\mathbb{F},\widehat{H})$. Let us lift $t$ and $f$ to $t_G \in C^0(\mathbb{F},G)$ and $f_G \in C^1(\mathbb{F},\widehat{G})$, respectively.
We have $d(t_G \cup f_G) = d(t_G) \cup f_G + t_G \cup d(f_G)$. By passing to cohomology classes we obtain
$\partial t \cup f = - t \cup \partial f $.
\end{proof}

\begin{subsec}
By taking the anticommutative diagram ~\eqref{e:h1_h2} with $\mathbb{F} = k$ and $\mathbb{F} = k_v,  v \in S$,
and arguing as in \ref{subsec:br_fin},
we obtain an anticommutative diagram
\be\label{e:h1_h2_fin}\xymatrix{
T(k_S) \ar[d]^{\partial^d_S}\ar@{}[r]|-{\times}  &\Sha_{S,\emptyset}^2(\widehat{T}) \ar[r]^-{\cup_S} &  \QQ / \ZZ \ar@{=}[d] \\
\Ch_S^1(H) \ar@{}[r]|-{\times}  &\Sha_{S,\emptyset}^1(\widehat{H}) \ar[u]^{\partial^u_S} \ar[r]^-{\cup_S} &  \QQ / \ZZ. \\
}\ee
Note that in diagram ~\eqref{e:h1_h2_fin}  we may write $\Ch_S^1(H)$ instead of $H^1(k_S,H)$, because
by the short exact sequence
$$
0\to \Br k\to\oplus\Br k_v \xrightarrow{\sum \inv_v} \QQ/\ZZ
$$
the image of the map $$\loc_S\colon H^1(k,H) \to H^1(k_S,H)$$
lies in the left kernel of the pairing
$$
H^1(k_S,H) \times \Sha^1_S(\widehat{H}) \xrightarrow{\cup_S} \QQ / \ZZ.
$$
\end{subsec}

Now we prove  that the map $\partial^d_S$ from diagram ~\eqref{e:h1_h2_fin} is exactly our obstruction $c_S$
and that it induces an isomorphism
$T(k_S) / \overline{T(k)} \to \Ch_S^1(H).$

\begin{lem}\label{l:iso_A_ch}
Let a pair of $k$-groups $(G,H)$ be such that $G$ is a (quasi-trivial) torus.
Set $T=G/H$.
We denote by
$\alpha\colon H^1(k_S,H) \to \Ch_S^1(H)$ the
canonical epimorphism, and by $\partial_S$ the map
$$ T(k_S) \xrightarrow{\partial_S} H^1(k_S,H),$$
obtained from the short exact sequence of groups of multiplicative type
$1 \to H \to G \to T \to 1$. Then
\begin{enumerate}
\item \label{l:iso_A_ch:it:partial_is_cs}$\alpha \circ \partial_S = c_S$;
\item\label{l:homomorphism} $c_S$ is a homomorphism;
\item \label{l:iso_A_ch:it:cs_is_surjective}$c_S$ is surjective;
 \item\label{l:iso_A_ch:it:cs_is_complete} $\Ker c_S = \overline{T(k)}$.
\end{enumerate}
\end{lem}

\begin{proof}
From the short exact sequence of groups of multiplicative type $$1 \to H
\to G \xrightarrow{\rho} T \to 0$$ we obtain a commutative diagram of abelian groups
\be\label{e:diag-G-T-H}
\xymatrix{
G(k)\ar@{^{(}->}[d]^-{\loc_S} \ar[r]^\rho &T(k)\ar@{^{(}->}[d]^-{\loc_S} \ar[r]^-\partial &H^1(k,H)\ar[d]^{\loc_S} \ar[r] &0\\
G(k_S) \ar[r]^-{\rho_S} &T(k_S) \ar[r]^-{\partial_S} &H^1(k_S,H) \ar[r] &0.
}
\ee
Since $H^{\mult} = H$, we have $c_S = \alpha \circ \tau_{S}\circ\nu_S$ (see \S\ref{ss:c_S}),
and ~\eqref{l:iso_A_ch:it:partial_is_cs} follows from the equality $\partial_S=\tau_S\circ\nu_S$.
Since $\alpha$ and $\partial_S$ are both homomorphisms, we see that $c_S$ is a homomorphism, which proves \eqref{l:homomorphism}.
Assertion ~\eqref{l:iso_A_ch:it:cs_is_surjective}  follows from the surjectivity of $\alpha$ and $\partial_S$.

We prove ~\eqref{l:iso_A_ch:it:cs_is_complete} by diagram chasing in diagram \eqref{e:diag-G-T-H}.
Since we know that $\overline{T(k)} \subseteq \Ker c_S$, it suffices to show that $\Ker c_S \subseteq \overline{T(k)}$.
Let $t_S \in\Ker c_S\subset T(k_S)$, i.e. $\alpha(\partial_S(t_S)) = 0$.
Then there exists $h \in H^1(k,H)$ such that $\partial_S(t_S) = \loc_S(h)$. We can find $t \in T(k)$
such that $\partial(t) = h$
(because $\partial$ is surjective).
We have
$$
\partial_S(t_S) = \loc_S(h) = \loc_S(\partial(t)) = \partial_S(\loc_S(t)),
$$ and therefore
$t_S-t \in \rho_S(G(k_S))$. Thus we have showed that $\Ker c_S \subseteq \rho_S(G(k_S)) + T(k)$, and
it suffices to prove that $ \rho_S(G(k_S)) + T(k) \subseteq  \overline{T(k)}$.
Since $G$ is quasi-trivial, it has the weak approximation property (cf. \cite{CT06}, Proposition 9.2),
and therefore $\overline{G(k)} = G(k_S)$.
We obtain:
\[
\rho_S(G(k_S)) + T(k) \subseteq \overline{\rho_S(G(k))}+T(k) \subseteq \overline{\rho(G(k))+T(k)}
= \overline{T(k)}.
\]
\end{proof}

The following lemma has been widely used, see e.g. \cite{Bor96}, Proof of Lemma 4.4, or \cite{Bor:cohomological-obstruction}, \S 3.5,
but we do not know a reference where it was stated, so we state and prove it here.

\begin{lem}\label{lem:abstract-nonsence}
Consider a commutative diagram of abelian groups with exact rows
$$
\xymatrix{
0\ar[r]  &A\ar[d]^{\lambda_A}\ar[r]^{\varphi}   &B\ar[d]^{\lambda_B}\ar[r]^{\psi}    &C\ar[d]^{\lambda_C} \\
0\ar[r]  &A'\ar[r]^{\varphi'}                   &B'\ar[r]^{\psi'}                    &C',
}
$$
then the induced sequence
$$
0\to\ker\lambda_A\labelto{\varphi_*} \ker\lambda_B\labelto{\psi_*}\ker\lambda_C
$$
is exact.
\end{lem}

\begin{proof}
We replace $C$ and $C'$ by $\im\psi$ and $\im\psi'$, resp., and apply the Snake Lemma.
\end{proof}

We proceed by proving that the map $\partial^u_S$ from diagram \eqref{e:h1_h2_fin} is an isomorphism.

\begin{lem}\label{l:iso_Sha_Be}
Let $1 \to H \to G \to T \to 1$ be a short exact sequence of
groups of multiplicative type such that $G$ is a quasi-trivial torus.
Let
\be\label{e:iso_sha12}\partial \colon \Sha_S^1(\widehat{H}) \to \Sha^2_S(\widehat{T})\ee
be the map induced by the connecting map obtained from the short exact sequence
$$0 \to \widehat{T} \to \widehat{G} \to\widehat{H} \to 0.
$$
Then $\partial$ is an isomorphism.
\end{lem}

\begin{proof}
From the short exact sequence $0 \to \widehat{T} \xrightarrow{i} \widehat{G} \to \widehat{H} \to 0$ we obtain
a commutative diagram with exact rows
$$
\xymatrix{
0 \ar[r]&H^1(k,\widehat{H}) \ar[d]^-{\loc_{S^\complement}} \ar[r]^-\partial &H^2(k,\widehat{T})\ar[d]^{\loc_{S^\complement}} \ar[r]^-{i}
&H^2(k,\widehat{G})\ar[d]^-{\loc_{S^\complement}} \\
0 \ar[r] &H^1(k_{S^\complement},\widehat{H}) \ar[r]^-{\partial_{S^\complement}} &H^2(k_{S^\complement},\widehat{T}) \ar[r]^-{i_{S^\complement}}
&H^2(k_{S^\complement},\widehat{G}),
}
$$
where $S^\complement=\sV(k)\smallsetminus S$.
Since the $k$-torus $G$ is quasi-trivial, by \cite{Sansuc}, (1.9.1), we have  $\Sha^2_S(\widehat{G})=0$.
Now by Lemma \ref{lem:abstract-nonsence} the induced homomorphism \eqref{e:iso_sha12} is an isomorphism.
\end{proof}

Using Lemmas ~\ref{l:iso_A_ch} and ~\ref{l:iso_Sha_Be}, we can rewrite
diagram ~\eqref{e:h1_h2_fin}.

\begin{lem}\label{l:h2_h1_fin}
Let $(G,H)$ be a pair of $k$-groups  such that $G$ is a quasi-trivial torus. Set $T=G/H$.
Then we have a commutative diagram
\be\label{e:t_main_1}\xymatrix{
T(k_S)/\overline{T(k)} \ar[d]^{-c_S}_{\cong} \ar@{}[r]|-{\times}  &\Sha_{S,\emptyset}^2(\widehat{T})  \ar[r]^-{\cup_S} &  \QQ / \ZZ \ar@{=}[d] \\
\Ch_S^1(H) \ar@{}[r]|-{\times}  &\Sha_{S,\emptyset}^1(\widehat{H}) \ar[u]^{\partial^u_S}_{\cong} \ar[r]^-{\cup_S} &  \QQ / \ZZ\, .
}\ee
In this diagram:
\begin{enumerate}
\item\label{l:h2_h1_fin:it:iso} Both vertical arrows are isomorphisms.
\item\label{l:h2_h1_fin:it:pairing} Both pairings are perfect pairings of finite abelian groups.
\end{enumerate}
\end{lem}
\begin{proof}
First note that we can write in diagram ~\eqref{e:h1_h2_fin} $T(k_S)/\overline{T(k)}$ instead of $T(k_S)$,
because $\overline{T(k)}$ lies in the left kernel of the top pairing by diagram ~\eqref{e:h2_br_fin},
and in the kernel of $\partial^d_S = c_S$ by Lemma ~\ref{l:iso_A_ch}.
By Lemmas ~\ref{l:iso_A_ch} and ~\ref{l:iso_Sha_Be}, diagram ~\eqref{e:t_main_1} is just a version of diagram ~\eqref{e:h1_h2_fin},
where we take $c_S$ with the negative sign in order to obtain a commutative diagram from the anticommutative diagram ~\eqref{e:h1_h2_fin}.
Also ~\eqref{l:h2_h1_fin:it:iso} follows from Lemmas ~\ref{l:iso_A_ch} and ~\ref{l:iso_Sha_Be}.

It remains to prove ~\eqref{l:h2_h1_fin:it:pairing}.
By Ono's lemma (cf. ~\cite{Ono}, Theorem 1.5.1) there is an exact sequence
$$ 1 \to B \to Q_1 \to T^m\times Q_2 \to 1$$
such that $m >0$ is an integer, $Q_1$ and $Q_2$ are quasi-trivial $k$-tori and $B$ is a finite abelian $k$-group.

Now we construct diagram ~\eqref{e:t_main_1} for the  pair of $k$-groups  $G=Q_1$ and $H=B$, then $Q_1/B=T^m\times Q_2$.
Using ~\eqref{l:h2_h1_fin:it:iso} we obtain
$$
\xymatrix{
\left(T(k_S)/\overline{T(k)}\right)^m \oplus Q_2(k_S)/\overline{Q_2(k)} \ar[d]^{-c_S}_{\cong} \ar@{}[r]|-{\times}
&\left(\Sha_{S,\emptyset}^2(\widehat{T}) \right)^m\oplus\Sha_{S,\emptyset}^2(\widehat{Q_2}) \ar[r]^-{\cup_S} &  \QQ / \ZZ \ar@{=}[d] \\
\Ch_S^1(B) \ar@{}[r]|-{\times} &\Sha_{S,\emptyset}^1(\widehat{B}) \ar[u]^{\partial^u_S}_{\cong} \ar[r]^-{\cup_S} &  \QQ / \ZZ \,.
}
$$

Since $Q_2$ is a quasi-trivial torus, it has the weak approximation property,\\ i.e. $Q_2(k_S)/\overline{Q_2(k)} = 0$.
Further, by ~\cite{Sansuc}, (1.9.1), we have $\Sha_S^2(\widehat{Q_2}) = 0$.
Thus we have a commutative diagram
\be\label{e:t_main_1_2}\xymatrix{
\left(T(k_S)/\overline{T(k)}\right)^m \ar[d]^{-c_S}_{\cong} \ar@{}[r]|-{\times} &\left(\Sha_{S,\emptyset}^2(\widehat{T}) \right)^m \ar[r]^-{\cup_S} &  \QQ / \ZZ \ar@{=}[d] \\
\Ch_S^1(B) \ar@{}[r]|-{\times} &\Sha_{S,\emptyset}^1(\widehat{B}) \ar[u]^{\partial^u_S}_{\cong} \ar[r]^-{\cup_S} &  \QQ / \ZZ \\
}\ee
where the vertical arrows are isomorphisms.

But by ~\cite{Sansuc}, the proof of Lemma 1.4, the bottom pairing in diagram ~\eqref{e:t_main_1_2} is a perfect pairing of finite abelian groups.
Therefore the top pairing:
$$
\xymatrix{
\left(T(k_S)/\overline{T(k)}\right)^m \ar@{}[r]|-{\times} &\left(\Sha_{S,\emptyset}^2(\widehat{T}) \right)^m \ar[r]^-{\cup_S} &  \QQ / \ZZ
}
$$
 is a perfect pairing of finite abelian groups, and clearly the same is true for the pairing
$$
\xymatrix{
T(k_S)/\overline{T(k)} \ar@{}[r]|-{\times} & \Sha_{S,\emptyset}^2(\widehat{T})  \ar[r]^-{\cup_S} &  \QQ / \ZZ
}.
$$
Now \eqref{l:h2_h1_fin:it:iso} gives us ~\eqref{l:h2_h1_fin:it:pairing}.
\end{proof}

Now we complete the proof of Theorem ~\ref{t:main_torus}.
\begin{proof}[Proof of Theorem \ref{t:main_torus}]
By composing the isomorphisms:
$$  \Sha_S^1(\widehat{H}) \isoto  \Sha^2_S(\widehat{T}) \isoto \Be_S(T)$$
from ~\eqref{e:iso_sha12} and ~\eqref{e:iso_sha2_be}, respectively,
we obtain  isomorphisms
\begin{align*}
\beta_S\colon &\Sha^1_S(\widehat{H}) \isoto \Be_S(T),\\
\beta_{S,\emptyset} \colon &\Sha^1_{S,\emptyset}(\widehat{H}):=\Sha^1_S(\widehat{H}) /\Sha^1_\emptyset(\widehat{H})
\isoto \Be_{S,\emptyset}(T),
\end{align*}
so we have proved  (i).
We  prove (ii) (a,b).
First we put diagram ~\eqref{e:h2_br_fin} on top of diagram ~\eqref{e:t_main_1} and obtain the following
commutative diagram, in which the vertical arrows are isomorphisms:
\be\label{e:t_main_2}
\xymatrix{
T(k_S)/\overline{T(k)} \ar@{=}[d] \ar@{}[r]|-{\times} &\Be_{S,\emptyset}(T) \ar[r]^-{\langle\, ,\, \rangle_S} & \QQ/\ZZ \ar@{=}[d] \\
T(k_S)/\overline{T(k)} \ar[d]^{-c_S}_{\cong} \ar@{}[r]|-{\times}
&\Sha_{S,\emptyset}^2(\widehat{T}) \ar[u]^{\theta_S}_{\cong} \ar[r]^-{\cup_S} &  \QQ / \ZZ \ar@{=}[d] \\
\Ch_S^1(H) \ar@{}[r]|-{\times} &\Sha_{S,\emptyset}^1(\widehat{H})  \ar[u]^{\partial^u_S}_{\cong} \ar[r]^-{\cup_S} &  \QQ / \ZZ .
}
\ee
By Lemma ~\ref{l:h2_h1_fin}\eqref{l:h2_h1_fin:it:pairing} the pairings
in the two lower rows of diagram ~\eqref{e:t_main_2} are perfect pairings of finite abelian groups.
Since all vertical arrows of diagram ~\eqref{e:t_main_2} are isomorphisms,
we see that the pairing in the top row is a perfect pairing of finite abelian groups as well.
We conclude that the diagram
$$
\xymatrix{
T(k_S)/\overline{T(k)} \ar[d]^{-c_S}_{\cong} \ar@{}[r]|-{\times} &\Be_{S,\emptyset}(T) \ar[r]^-{\langle\, ,\, \rangle_S} & \QQ/\ZZ \ar@{=}[d] \\
\Ch_S^1(H) \ar@{}[r]|-{\times} &\Sha_{S,\emptyset}^1(\widehat{H})  \ar[u]^{\beta_S}_{\cong} \ar[r]^-{\cup_S} &  \QQ / \ZZ \\
}
$$
satisfies (a,b), as required.
\end{proof}

\begin{proof}[Proof of Theorem ~\ref{t:dual_t}]
Let $H$ be a $k$-group of multiplicative type.
Since the Galois module $\widehat{H}$ is a quotient of a permutation Galois module,
there exists an embedding $H \hookrightarrow G$  into a quasi-trivial $k$-torus $G$.
By applying Lemma ~\ref{l:h2_h1_fin} to the pair $(G,H)$ we obtain the desired duality.
\end{proof}


\section{Proofs: the general case}\label{s:general_case}
In this section we  prove our results in the general case
by using morphisms of pairs $(G,H)$
which preserve both the Brauer-Manin obstruction
and  our cohomological obstruction $c_S$. Using
such morphisms we  reduce the proof of Theorem ~\ref{t:main_intro}
to the case when $G$ is a torus, which was dealt with
in the previous section.

The following theorem is the  main result of this paper.
\begin{thm}\label{t:main_gen}
Let $(G,H)$ be a pair of $k$-groups as in \ref{subsec:pairs-of-groups} (with quasi-trivial $G$) over a number field $k$,
and let $S \subset \cV(k)$ a finite set of places of $k$. Then
\begin{enumerate}
\item\label{t:main_gen:it:be} There are canonical isomorphisms
$$\beta_S\colon \Sha^1_S(\widehat{H}) \isoto \Be_S(X),$$
$$ \beta_{S,\emptyset}\colon \Sha^1_{S,\emptyset}(\widehat{H}) \isoto \Be_{S,\emptyset}(X).$$
\item\label{t:main_gen:it:main} The following canonical diagram commutes:
\be\label{e:main}
\xymatrix{
X(k_S) \ar[d]^-{-c_S}  \ar@{}[r]|{\times} & \Be_{S,\emptyset}(X) \ar[r]^-{\langle\,,\,\rangle_S} &\QQ / \ZZ \ar@{=}[d] \\
 \Ch^1_S(k,H^{\mult}) \ar@{}[r]|{\times} & \Sha^1_{S,\emptyset}(\widehat{H}) \ar[u]^{\beta_{S,\emptyset}}_{\cong} \ar[r]^-{\cup_S} &  \QQ / \ZZ
}
\ee
In diagram ~\eqref{e:main} the bottom pairing is a perfect pairing of finite groups.
\end{enumerate}
\end{thm}

\def\Brr{\mathrm{Br}}
\def\arm{{\mathrm a}}

In order to prove Theorem ~\ref{t:main_gen} we shall construct auxiliary pairs of $k$-groups and morphisms of pairs,
similar to ~\cite{Bor96}.

\begin{subsec}\label{subsec:G,H,X,Y}
Let  $(G,H)$ be a pair of $k$-groups  as in \ref{subsec:pairs-of-groups} (with quasi-trivial $G$) over a number field $k$.
Choose an embedding $i\colon H^{\mult} \to Q$
into a quasi-trivial torus $Q$.
Consider the embedding
$$i_*\colon H \to G\times_k Q, \qquad h \mapsto (h,i(\multmap(h))),$$
where $\multmap\colon H \to H^{\mult}$ is the canonical epimorphism.
Set $G_Y=G\times_k Q$, $H_Y=i_*(H)$. The pair $(G_Y,H_Y)$ defines a homogeneous space $Y=G_Y/H_Y=(G\times_k Q)/i_*(H)$ with a marked point $y^0$.
 The projection map $\pi\colon G_Y = G \times Q \to G$ is surjective and satisfies $\pi(H_Y) = H$, and therefore it defines a
 morphism of pairs as in \ref{subsec:pairs-of-groups} $\pi\colon (G_Y,H_Y)\to(G,H)$,
which in turn defines a morphism of varieties with marked points $\pi_*\colon (Y,y^0)\to (X,x^0)$.
Note that the map $H_Y^{\mult}\to G_Y^{\tor}$ is injective and that $(Y,\pi_*)$ is a torsor over $X$ under $Q$.
Since $Q$ is a quasi-trivial torus, we see that for any field $\mathbb{F}$ containing $k$ the map
$\pi_* \colon Y(\mathbb{F}) \to X(\mathbb{F})$
is surjective.
\end{subsec}

\begin{lem}\label{l:Y_prop}
With notation and assumptions of \ref{subsec:G,H,X,Y} we have:
 \begin{enumerate}
\item\label{l:Y_prop:it:U} $U(\overline{Y}) \cong \widehat{G_Y^{\tor}/ H_Y^{\mult}}$.
\item\label{l:Y_prop:it:pic} $\Pic \overline{Y} = 0$.
\item\label{l:Y_prop:it:bra} There is a canonical functorial isomorphism
$$\theta \colon H^2(k,U(\overline{Y})) \isoto \Bra Y \isoto\Brr_{y^0}Y$$
\end{enumerate}
where $\overline{Y} = Y \times_k \bar{k}.$
\end{lem}

\begin{proof}
Since the map $H_Y^{\mult} \to G_Y^{\tor}$ is injective,
~\eqref{l:Y_prop:it:U} and ~\eqref{l:Y_prop:it:pic} follow from Proposition ~\ref{l:H_k_pic_u}.
It remains to prove ~\eqref{l:Y_prop:it:bra}. The Hochschild-Serre spectral sequence
$$
 H^p(\Gal(\kbar/k),H^q(\overline{Y},\mathbb{G}_m)) \Rightarrow H^n(Y,\mathbb{G}_m)
$$
gives rise to an exact sequence
$$ H^0(k,\Pic \overline{Y}) \to H^2(k,U(\overline{Y})) \xrightarrow{\theta_\arm} \Bra Y \to H^1(k,\Pic \overline{Y})$$
(see ~\cite{Sansuc}, Lemma 6.3(ii)).
We see from ~\eqref{l:Y_prop:it:pic} that $\theta_\arm$ is an isomorphism.
Composing $\theta_\arm$ with the canonical isomorphism $\Bra Y \isoto\Brr_{y^0}Y$,
we obtain the desired isomorphism $\theta$.
\end{proof}

\begin{subsec}
Consider the canonical epimorphism $\multmap\colon G_Y \to G_Y^{\mult} = G_Y^{\tor}$.
Then the following diagram commutes and has injective horizontal arrows and surjective vertical arrows:
\be \xymatrix{
H_Y \ar[d]^{\multmap} \ar[r] &G_Y \ar[d]^{\multmap}\\
H_Y^{\mult} \ar[r]^{j} & G_Y^{\tor}.
}\ee\label{e:hm_inj}

 We  construct a new pair $(G_Z,H_Z)$ as follows:
 $G_Z=G_Y^\tor$, $H_Z=j(H_Y^\mult)$.
 Set $Z=G_Z/H_Z$, it is a $k$-torus, we denote by $z^0$ its identity element.
We have a morphism of pairs $\multmap\colon (G_Y,H_Y)\to (G_Z, H_Z)$ as in \ref{subsec:pairs-of-groups}
and the induced morphism of homogeneous spaces
$\mu_*\colon (Y,y^0)\to (Z,z^0)$.
Thus we have  diagrams
$$
\xymatrix{
(G_Y, H_Y) \ar[d]^{\multmap}\ar[r]^{\pi} & (G,H) & &(Y,y^0) \ar[d]^{\multmap_*}\ar[r]^{\pi_*} & (X,x^0)\\
(G_Z,H_Z)                               &   & &(Z,z^0)
}
$$
Since $G_Z$ is a torus, we know that Theorem ~\ref{t:main_gen} is true for $Z$ (this is  Theorem ~\ref{t:main_torus}).
So all we have to do is to prove that both ${\multmap}$ and ${\pi}$ preserve $\mWS$ and $c_S$.
\end{subsec}

\begin{lem}\label{l:premain1}
The following diagrams commute and all the vertical arrows marked with $(\cong)$ are isomorphisms.
\be\label{e:ch_sha_iso}
\xymatrix{
X(k_S) \ar[r]^{c_S} & \Ch^1_S(H^{\mult}) &  & \Ch^1_S(H^{\mult})
\ar@{}[r]|-{\times} &\Sha^1_S(\widehat{H})\ar[d]_\cong^{\widehat{{\pi}}} \ar[r]^{\cup_S} &\QQ / \ZZ \ar@{=}[d]\\
Y(k_S) \ar[u]_{{\pi_*}} \ar[d]^{{\multmap_*}} \ar[r]^{c_S} & \Ch^1_S(H_Y^{\mult})\ar[u]^\cong_{{\pi}^{\multmap}}\ar[d]_\cong^{{\multmap}^{\multmap}}
&  & \Ch^1_S(H_Y^{\mult})\ar[u]^\cong_{{\pi}^{\multmap}}\ar[d]_\cong^{{\multmap}^{\multmap}}
                 \ar@{}[r]|-{\times} &\Sha^1_S(\widehat{H_Y}) \ar[r]^{\cup_S} &\QQ / \ZZ\\
Z(k_S) \ar[r]^{c_S} & \Ch^1_S(H_Z^{\mult}) &  & \Ch^1_S(H_Z^{\mult})
\ar@{}[r]|-{\times} &\Sha^1_S(\widehat{H_Z})
     \ar[u]^\cong_{\widehat{{\multmap}}}\ar[r]^{\cup_S} &\QQ / \ZZ\ar@{=}[u]\\
}
\ee
\end{lem}

\begin{proof}
The lemma follows from the functoriality of $c_S$ and from the fact that ${\pi}^{\multmap}\colon H_Y^{\mult} \to H^{\mult}$
and ${\multmap}^{\multmap}\colon H_Y^{\mult} \to H_Z^{\mult}$ are isomorphisms.
\end{proof}

\begin{lem}\label{l:pre_main2}
The following diagram commutes and all the vertical arrows marked with $(\cong)$ are isomorphisms.
\be\label{e:bm_func}
\xymatrix{
X(k_S) \ar@{}[r]|-{\times} &\Be_S(X)\ar[d]_\cong^{{{\pi}}^*} \ar[r]^{\langle\,,\,\rangle_S} &\QQ / \ZZ \ar@{=}[d]\\
Y(k_S) \ar[u]_{{\pi_*}} \ar[d]^{{\multmap_*}} \ar@{}[r]|-{\times} &\Be_S(Y) \ar[r]^{\langle\,,\,\rangle_S} &\QQ / \ZZ\\
Z(k_S) \ar@{}[r]|-{\times} &\Be_S(Z) \ar[u]^\cong_{{{\multmap}}^*}\ar[r]^{\langle\,,\,\rangle_S} &\QQ / \ZZ\ar@{=}[u]\\
}
\ee
\end{lem}

\begin{proof}
The commutativity of the diagram follows from the functoriality of $\langle\,,\,\rangle_S$.
We prove that ${\pi}^*$ is an isomorphism.
Since $Y$ is a torsor over $X$ under $Q$,  by ~\cite{Sansuc}, (6.10.3), there is an exact sequence
$$
\Pic Q \to \Brn X \xrightarrow{{\pi}^*} \Brn Y \to \Brr_e Q,
$$
hence an exact sequence
$$
\Pic Q \to \Brr_{x^0} X \xrightarrow{{\pi}^*} \Brr_{y^0} Y \to \Brr_e Q.
$$
By ~\cite{Sansuc}, Lemma 6.9(ii), we have $\Pic Q \cong H^1(k,\widehat{Q})=0$ (because $Q$ is quasi-trivial)
and  $\Brr_e Q \cong H^2(k,\widehat{Q})$.
We obtain a commutative diagram
\be\label{e:br_iso1}
\xymatrix{
0 \ar[r] & \Brr_{x^0} X \ar[r]^{{\pi}^*}\ar[d]^{\loc_{S^\complement}} \ar[r]
&\Brr_{y^0} Y \ar[d]^{\loc_{S^\complement}}\ar[r] & H^2(k,\widehat{Q}) \ar[d]^{\loc_{S^\complement}} \\
0 \ar[r] & \bigoplus_{v \in S^{\complement}} \Brr_{x^0} X_{k_v} \ar[r]^{{\pi}^*} \ar[r]
&\bigoplus_{v \in S^{\complement}} \Brr_{y^0} Y_{k_v} \ar[r] & H^2(k_{S^{\complement}},\widehat{Q})
}
\ee
where $S^\complement=\sV(k)\smallsetminus S$.
By ~\cite{Sansuc}, (1.9.1), we have $\Sha^2_S(\widehat{Q}) = 0$.
By Lemma \ref{lem:abstract-nonsence} our diagram ~\eqref{e:br_iso1} induces an isomorphism
$$
{\pi}^* \colon \Be_S(X) \to \Be_S(Y).
$$

It is remains to prove that ${\multmap}^*$ is an isomorphism.
From the construction of $(G_Z,H_Z)$ we obtain a commutative diagram with injective horizontal arrows
and bijective vertical arrows
$$  \label{e:bra_iso}
\xymatrix{
H_Y^{\mult}\ar[d]_{\cong}^{{\multmap}} \ar[r] &G_Y^{\tor}\ar[d]_{\cong}^{{\multmap}}\\
H_Z^{\mult}                    \ar[r] &G_Z^{\tor}
}
$$
From this diagram
we obtain an isomorphism ${\multmap}_*\colon G_Y^{\tor}/H_Y^{\mult} \to G_Z^{\tor}/H_Z^{\mult}$.
By  Lemma ~\ref{l:Y_prop} and Diagram (8) in \S 4.4 we have canonical isomorphisms
\begin{gather*}
\Brr_{z^0} Z \cong H^2(k,U(\overline{Z})) \cong H^2(k,\widehat{G_Z^{\tor}/H_Z^{\mult}}),\\
\Brr_{y^0} Y \cong H^2(k,U(\overline{Y})) \cong H^2(k,\widehat{G_Y^{\tor}/H_Y^{\mult}}).
\end{gather*}
We obtain a commutative diagram
$$ \label{e:Br-H2}
\xymatrix{
\Brr_{z^0} Z\ar[r]^-\cong\ar[d]^{\mu^*}         & H^2(k,\widehat{G_Z^{\tor}/H_Z^{\mult}})\ar[d]^{\mu^*}   \\
\Brr_{y^0} Y \ar[r]^-\cong                   &H^2(k,\widehat{G_Y^{\tor}/H_Y^{\mult}}),
}
$$
where the right-hand vertical arrow is an isomorphism.
It follows that the left-hand vertical arrow in this diagram is an isomorphism,
hence the homomorphism ${\multmap}^*\colon\Be_S(Z)\to\Be_S(Y)$ is an isomorphism.
\end{proof}

Now we complete the proof of Theorem ~\ref{t:main_gen}.

\begin{proof}[Proof of Theorem ~\ref{t:main_gen}]
We  prove ~\eqref{t:main_gen:it:be}.
Since $G_Z$ is a quasi-trivial torus, by Theorem ~\ref{t:main_torus}\eqref{t:main_gen:it:be}  we have an isomorphism
$$
\beta^Z_S\colon \Sha^1_S(\widehat{H_Z}) \isoto \Be_S(Z).
$$
Using diagrams ~\eqref{e:ch_sha_iso} and ~\eqref{e:bm_func}, we obtain a diagram with bijective arrows
$$
\xymatrix{
\Sha^1_S(\widehat{H}) \ar[r]_\cong^{\widehat{{\pi}}} &\Sha^1_S(\widehat{H_Y})
        &\Sha^1_S(\widehat{H_Z})\ar[l]^\cong_{\widehat{{\multmap}}} \ar[d]_\cong^{\beta^Z_S} \\
\Be_S(X) \ar[r]_\cong^{{\pi}^*} &\Be_S(Y) &\Be_S(Z) \ar[l]^\cong_{{\multmap}^*}
}
$$
This diagram gives us the required isomorphism
 $$\beta_S \colon \Sha^1_S(\widehat{H}) \isoto \Be_S(X).$$
The map $\beta_S$ induces an isomorphism
$$\beta_{S,\emptyset} \colon \Sha^1_{S,\emptyset}(\widehat{H}) \isoto \Be_{S,\emptyset}(X).$$

We prove ~\eqref{t:main_gen:it:main}.
By Theorem ~\ref{t:dual_t} the bottom pairing in diagram  ~\eqref{e:main} is a perfect pairing of finite groups.
We prove the commutativity of this diagram
by a diagram chase in diagrams  ~\eqref{e:ch_sha_iso} and ~\eqref{e:bm_func}.

Let $x_S \in X(k_S)$ and $t \in \Sha^1_S(\widehat{H})$.
Since $Y\to X$ is a torsor under the quasi-trivial torus $Q$,
there exists $y_S \in Y(k_S)$ such that ${\pi}(y_S) = x_S$.
We set $z_S = {\multmap}(y_S)\in Z(k_S)$.
We set $t_Y := \widehat{{\pi}}(t) \in \Sha^1_S(\widehat{H_Y})$
and $t_Z = {\widehat{{\multmap}}}^{-1}(t_Y) \in \Sha^1_S(\widehat{H_Z})$.
By diagrams ~\eqref{e:ch_sha_iso} we have
$$ \cup_S(-c_S(x_S) , t) = \cup_S(-c_S(y_S) , t_Y) =  \cup_S(-c_S(z_S) , t_Z).$$
By Theorem ~\ref{t:main_torus} we have
$$ \cup_S(-c_S(z_S) , t_Z) = \langle z_S, \beta^Z_S(t_Z) \rangle_S.$$
Set $b_Z=\beta^Z_S(t_Z)$, then
$$
\cup_S(-c_S(z_S) , t_Z)=\langle z_S, b_Z \rangle_S.
$$
Set $b_Y=\multmap^*(b_Z)\in \Be_S(Y)$, \ $b_X=(\pi^*)^{-1}(b_Y)\in\Be_S(X)$,
then by diagram ~\eqref{e:bm_func} we have
$$
\langle z_S, b_Z \rangle_S=\langle y_S, b_Y \rangle_S=\langle x_S, b_X \rangle_S.
$$
Since $b_X=\beta_S(t)$,
we obtain that
$$\cup_S(-c_S(x_S) , t) = \langle x_S, \beta_S(t) \rangle_S,$$
as required.
\end{proof}

\end{document}